\documentclass[reqno]{amsart}

\usepackage{amsmath,amssymb,amsthm}

\usepackage{tikz}
\usepackage{caption}
\usepackage{graphicx}
\usepackage{graphics}

\usepackage{hyperref}

\newtheorem*{thm}{Theorem}
\newtheorem{lemma}{Lemma}

\usepackage{url}
\usepackage{xcolor}

\begin{document}

\title[]{Nonlinear recursions on the reals\\and a problem of Graham}

\author[]{Stefan Steinerberger}

\address{Department of Mathematics, University of Washington, Seattle, WA 98195, USA}
 \email{steinerb@uw.edu}

\thanks{The author is partially supported by the NSF (DMS-2123224).}

\begin{abstract} We study sequences $(x_n)_{n=1}^{\infty}$ of reals given by $x_{n+1} = f(x)$ where
$$f(x) = x - \sum_{i=1}^{m} \frac{\alpha_i}{x - \beta_i},$$
where $\alpha_1, \dots, \alpha_m \in \mathbb{R}_{>0}$ and $\beta_1, \dots, \beta_m \in \mathbb{R}$ are arbitrary.
A special case is $x_{n+1} = x_n - 1/x_n$ due to Ronald Graham for which Chamberland \& Martelli showed that the dynamics is  chaotic (topologically conjugate to the doubling map). We prove that the general nonlinear recursion, despite being potentially chaotic, is effective at ensuring that most iterates end up close to one of the poles $\beta_i$ relatively quickly. More precisely, for a positive proportion of initial values $x \in \mathbb{R}$, the sequence gets very close (distance $\lesssim |x|^{-1}$) to one of the poles $\beta_i$ within a relatively small ($\lesssim x^2$) number of iteration steps.  
\end{abstract}
\maketitle

\vspace{-10pt}

\section{Introduction}
\subsection{Graham's problem} We were motivated by the following problem described by Ronald Graham in a talk given at UCLA in August 2000 \cite{chamberland}. Consider the sequence $(x_n)_{n=0}^{\infty}$ of real numbers defined by 
$$ x_{n+1} = x_n - \frac{1}{x_n}, \quad x_0 = 2.$$
\begin{quote}
\textbf{Question.} Is the sequence unbounded?
\end{quote}
It is clear that if $x_{n} \gg 1$, then the sequence is going to be monotonically decreasing while for $x_n \ll -1$ it is monotonically increasing: when far away from the origin, the sequence tends towards the origin. The dynamics close to the origin is complicated. It is also quickly seen, see Fig. 1, that
iterating the function leads to chaos.
\begin{center}
\begin{figure}[h!]
\begin{tikzpicture}
\node at (0,0) {\includegraphics[width = 0.44\textwidth]{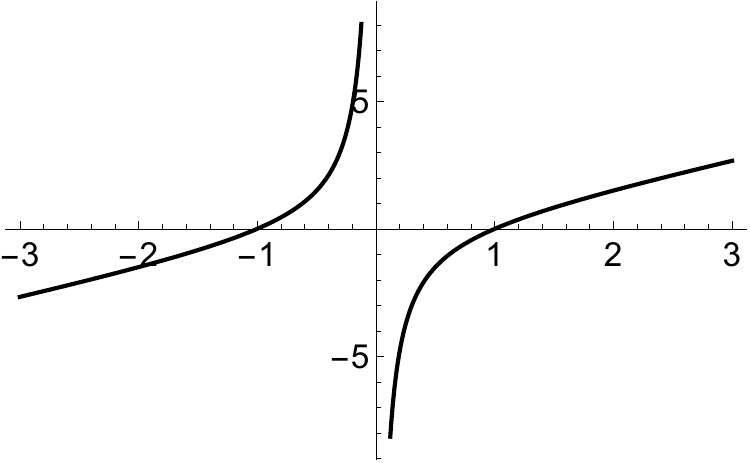}};
\node at (6.5,0) {\includegraphics[width = 0.44\textwidth]{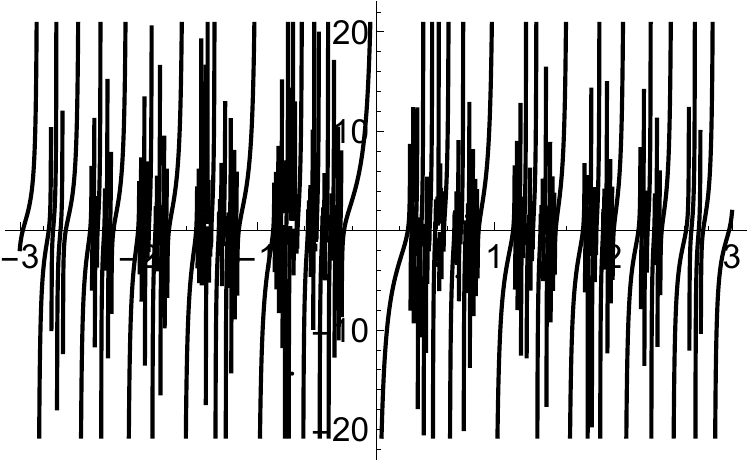}};
\end{tikzpicture}
\caption{$f(x) = x-1/x$ and $f^{(8)}(x)$, both on $[-3,3]$.}
\end{figure}
\end{center}
This chaotic behavior was nicely explained by Chamberland \& Martelli \cite{chamberland}.
\begin{thm}[Chamberland \& Martelli \cite{chamberland}] The map $f(x) = x - 1/x$ is topologically conjugate to the doubling map $2x \mod 1$ on the interval $[0,1)$. \end{thm}
Their result has a number of immediate consequences. Understanding the orbit of any given initial value is roughly as complicated as understanding the distribution of binary digits of a given real number and, as a consequence, Graham's question is probably not going to get answered soon. However, the Theorem also implies that
for almost all initial conditions, the orbit is dense on the real line (and it is, of course, a reasonable guess that $2$ is a `typical' initial value).

\subsection{Heuristic.} One could rephrase Graham's question a little bit: asking whether the sequence is unbounded is really equivalent to asking whether it gets arbitrarily cose to the origin. Unboundedness can only arise by exploiting the singularity in the origin (see Fig. 2 for a cartoon of what Graham's iteration approximately does). So we can ask a very related question.\\

\begin{quote}
\textbf{Question.} For $\varepsilon > 0$ and $x_0 \in \mathbb{R}$, how many iterations of $f(x) = x-1/x$ does it usually take to end up in the interval $[-\varepsilon, \varepsilon]$?\\
\end{quote}

\begin{center}
\begin{figure}[h!]
\begin{tikzpicture}
\draw [thick, <->] (0,0) -- (8,0);
\draw [very thick] (3,-0.1) -- (3,0.1);
\node at (3, -0.3) {$-1$};
\draw [very thick] (4,-0.1) -- (4,0.1);
\node at (4, -0.3) {$0$};
\draw [very thick] (5,-0.1) -- (5,0.1);
\node at (5, -0.3) {$1$};
\draw[thick,->] (7.5, 0) to[out=110, in=0] (7.25, 0.4) to[out=180, in=70] (7,0.1);
\draw[thick,->] (7.25, 0) to[out=110, in=0] (7, 0.4) to[out=180, in=70] (6.75,0.1);
\draw[thick,->] (7, 0) to[out=110, in=0] (6.6, 0.4) to[out=180, in=70] (6.2,0.1);
\draw[thick,->] (6.7, 0) to[out=110, in=0] (6.3, 0.4) to[out=180, in=70] (5.9,0.1);
\draw[thick,->] (6.3, 0) to[out=110, in=0] (5.9, 0.4) to[out=180, in=70] (5.5,0.1);
\draw[thick,->] (6, 0) to[out=110, in=0] (5.4, 0.4) to[out=180, in=70] (4.8,0.1);
\draw[thick,->] (5.75, 0) to[out=110, in=0] (5.1, 0.4) to[out=180, in=70] (4.6,0.1);
\draw[thick,->] (0.5,0) to[out=70, in=180] (0.8, 0.4) to[out=0, in=110] (1.1,0.1);
\draw[thick,->] (0.75,0) to[out=70, in=180] (1.05, 0.4) to[out=0, in=110] (1.4,0.1);
\draw[thick,->] (1,0) to[out=70, in=180] (1.3, 0.4) to[out=0, in=110] (1.6,0.1);
\draw[thick,->] (1.3,0) to[out=70, in=180] (1.7, 0.4) to[out=0, in=110] (2.1,0.1);
\draw[thick,->] (1.6,0) to[out=70, in=180] (2.1, 0.4) to[out=0, in=110] (2.5,0.1);
\draw[thick,->] (2,0) to[out=70, in=180] (2.4, 0.4) to[out=0, in=110] (2.8,0.1);
\draw[thick,->] (2.5,0) to[out=70, in=180] (3, 0.4) to[out=0, in=110] (3.4,0.1);
\draw[thick,->] (2.7,0) to[out=70, in=180] (3.2, 0.4) to[out=0, in=110] (3.6,0.1);
\draw[thick,->] (3.7,0) to[out=60, in=180] (4.7, 0.8) to[out=0, in=130] (5.2,0.5);
\draw[thick,->] (3.8,0) to[out=70, in=180] (4.9, 0.8) to[out=0, in=120] (6,0.5);
\draw[thick,->] (3.9,0) to[out=70, in=180] (6, 1) to[out=0, in=120] (7.5,0.5);
\draw[thick,->] (4.1,0) to[out=110, in=0] (2, 1) to[out=180, in=60] (0,0.5);
\draw[thick,->] (4.2,0) to[out=110, in=0] (2.5, 0.9) to[out=180, in=60] (1,0.5);
\draw[thick,->] (4.3,0) to[out=110, in=0] (3, 0.8) to[out=180, in=60] (2,0.5);
\end{tikzpicture}
\caption{A sketch of iterating $f(x) = x-1/x$. Points bigger than 1 get sent to point in $[0,1]$ which then moves to a point $\leq -1$ which moves to $[-1,0]$ which is sent to a point $\geq 1$ and so on.}
\end{figure}
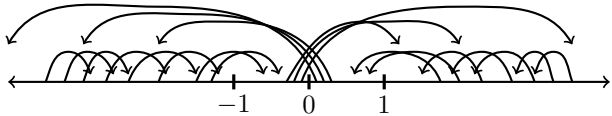
\end{center}

\vspace{-10pt}

To get a feeling for the main result we quickly sketch a heuristic for $f(x) = x - 1/x$. Let $0 < \varepsilon \ll 1$ and suppose we start at distance $|x_0| \sim 1/\varepsilon$ from the origin: how many iterations would one expect does it take until one hits the interval $[-\varepsilon, \varepsilon]$ for the first time? If we are in a point $|x_0| \gg 1$, it takes approximately $\sim x_0^2$ iteration steps to reach the interval $[-1,1]$ for the first time (see Lemma 1 for a precise formulation of this). Then, assuming that everything is fairly random/chaotic, there's a $\sim \varepsilon$ chance of ending in $[-\varepsilon,\varepsilon]$ and a $1-\varepsilon$ chance of missing it. If we end up in a random spot in $[\varepsilon, 1]$, the map transports us, on average, to distance
$$ \int_{\varepsilon}^{1} \frac{dy}{y} \sim \log\left(\frac{1}{\varepsilon}\right) \qquad \mbox{from the origin, requiring another} \qquad  \int_{\varepsilon}^{1} \frac{dy}{y^2} \sim \frac{1}{\varepsilon}$$
steps to get back to $[-1,1]$ for another chance to hit the target interval $[-\varepsilon, \varepsilon]$. This random heuristic suggests that, starting from $|x_0| \sim 1/\varepsilon$ it would take approximately $\varepsilon^{-2}$ steps to first get to the interval $[-1,1]$. Then we would need another $\sim 1/\varepsilon$ attempts to hit the target $[-\varepsilon, \varepsilon]$ and each failed attempts costs, on average, $1/\varepsilon$ additional steps to get back to $[-1,1]$ and try again for a total of $\sim \varepsilon^{-2} \sim |x_0|^2$ steps. Naturally, this heuristic is far from rigorous, everything is deterministic and nothing is random. The main result shows that this scaling is nonetheless accurate.

\subsection{Main Result}
We will show that such recursions are effective at ensuring that iterates get close to one of the poles. 
\begin{thm} Let  $m \in \mathbb{N}$, $\alpha_1, \dots, \alpha_m \in \mathbb{R}_{>0}$ and $\beta_1, \dots, \beta_m \in \mathbb{R}$ and consider
 $$ f(x) = x - \sum_{i=1}^{m} \frac{\alpha_i}{x - \beta_i}.$$
 
There exist $c_1, c_2 > 0$ (depending on $f$) such that the set
 $$X = \left\{ x \in \mathbb{R}: \min_{1 \leq n \leq c_1 |x|^2} ~~ \min_{1 \leq i \leq m} ~~ |f^{(n)}(x) - \beta_i|  \leq \frac{1}{|x|} \right\}$$
 has positive density in the sense of $|X \cap [-y,y]| \geq c_2 y$ for all sufficiently large $y$.
\end{thm}
One reason why the result is interesting is as follows: starting in $x \gg 1$, it takes at least $\sim c \cdot  x^2$ iterations to end up at distance 1 from one of the poles for the first time. The chance of ending up distance $\varepsilon$ the first time one gets within distance $\leq 1$ to the poles is small: this means that if one misses the interval, one nonetheless has a decent chance of ending up at distance $\varepsilon$ within another $\sim x^2$ steps; this is, in a sense, indicative of very structured mixing close to the origin.\\
The result is optimal up to constants in different ways. At least $\sim x^2$ iterations are required for the iterates to even get close to the origin; in particular, if $c_1$ is chosen too small (depending on $f$), the set $X$ would be empty. Approximating at rate $1/|x|$ is also optimal if we want the set to have positive density. The proof is not too involved and uses two facts that nicely play together: the first is that the dynamical system is measure-preserving in the sense that $|f^{-1}(A)| = |A|$. This is equivalent to a celebrated identity in the study of closed-form solutions of integrals, Glasser's theorem \cite{glasser}. In the case of Graham's recursion $f(x) = x - 1/x$, this is implied by the \textit{Cauchy-Schl\"omilch identity} \cite{cauchy, liou} for absolutely integrable $g$
$$ \int_{\mathbb{R}} g(x) dx = \int_{\mathbb{R}} g(x-1/x) dx.$$
 The second ingredient comes from the fact that points close to one of the poles gets transported very far away in the next iteration and remain far away for a long time. This leads to a disjointness of pre-images that can be exploited.\\
  It stands to reason that the confluence of these two rare factors suggests that Graham's recursion and, more generally, dynamical systems of this type might have other interesting properties. It would be interesting if more refined consequences of the heuristic could be made precise: for example, the heuristic suggests that almost all $x \in \mathbb{R}$ should end up within distance $1/|x|$ of one of the poles within, say, $x \cdot (\log{x})^2$ iterations. It might also be interesting to understand whether higher dimensional analogues exist; a generalization of the Glasser Theorem due to Aomoto \& Forrester \cite{ao} could be relevant in this regard.

\section{Proof}
We assume that $m, \alpha_1, \dots, \alpha_m > 0$ and $\beta_1, \dots, \beta_m$ are fixed and that
 $ f(x) = x - \sum_{i=1}^{m} \alpha_i/(x - \beta_i).$
 It is easy to see that $\alpha_i > 0$ is necessary, consider the example $f(x) = x + 1/x$.
We assume, without loss of generality, that the $\beta_i$ are all distinct. We also assume (wlog) that the ordering $\beta_1 < \beta_2 < \dots < \beta_m$. Unspecified constants will depend on the function $f$ in ways that could be made explicit.

\subsection{Slow Movement Lemma}
The first step consists in showing that if one starts far away from the poles, say distance $\sim x$ to the nearest pole, then it takes many ($\sim x^2$) steps to approach the poles.
\begin{lemma}
There exist constants $c_1, c_2 > 0$ (depending only on $f$) such that if 
$$ \min_{1 \leq i \leq m} |x- \beta_i| \geq c_1,$$
then
$$ \min_{1 \leq j \leq c_2 x^2} ~~ \min_{1 \leq i \leq m} |f^{(j)}(x) - \beta_i| \geq \frac{1}{2}  \min_{1 \leq i \leq m} |x- \beta_i|.$$
\end{lemma}
\begin{proof}
By choosing $c_1 \geq 2 (\beta_m - \beta_1)$, we can ensure that $x$ is not `between' the poles but either smaller than the smallest pole or larger than the largest pole. We will only go through with the case $x > \beta_m$, the case $x < \beta_1$ is analogous. 
If $x > \beta_m$, then $x - \beta_i > 0$ and, since $\alpha_i > 0$, we deduce that
$$  f(x) = x - \sum_{i=1}^{m} \frac{\alpha_i}{x - \beta_i} < x,$$
the iterates are moving closer to the poles.
Introducing
$$ g(x) = x - \frac{\sum_{i=1}^{m} \alpha_i}{x - \beta_m},$$
it is easy to see that 
$$ x > \beta_m \implies g(x) < f(x).$$
It thus suffices to study iterations of $g$ and show that these cannot move too quickly to the poles. Using translation invariance, we may assume that $x_0 = x$ and $\beta_m = 0$. It therefore suffices to study the recursion 
$$ x_{n+1} = x_n -  \frac{\sum_{i=1}^{m} \alpha_i}{x_n}, \qquad x_0 = x > 0.$$
We note that, as long as $x_n \geq x/2$, we have
$$ x_{n+1} \geq x_n - \frac{2 \sum_{i=1}^{m} \alpha_i}{x}.$$
This means it is going to take at least 
$$ \frac{x/2}{\frac{2 \sum_{i=1}^{m} \alpha_i}{x}} =  \left(4 \sum_{i=1}^{m} \alpha_i \right)^{-1} x^2 $$
steps for iterations of $g$ to move from $x$ to a number smaller than $x/2$ and this is the desired statement.
\end{proof}

\subsection{Disjointness Lemma}
The second step in the argument is a disjointness lemma. We introduce, for $\varepsilon > 0$, the set of all real numbers
with the property that the $k-$th iterate is $\varepsilon-$close to one of the poles
$$ I_k = \left\{x \in \mathbb{R}: \min_{1 \leq i \leq m} | f^{(k)}(x)- \beta_i| \leq \varepsilon \right\}.$$
$I_k$ depend on both $k$ and $\varepsilon$, we suppress $\varepsilon$ in the notation.
Note that Lemma 1 already implies that these sets $I_k$ have compact support. One could deduce that, for some constant $c > 0$ depending only on $f$ that the set is contained in $I_k \subseteq [-c \sqrt{k}, c\sqrt{k}]$. The next Lemma shows that
the two sets $I_k$ and $I_{\ell}$ are disjoint, $I_k \cap I_{\ell} = \emptyset$, when $k$ and $\ell$ are too close to each other.
\begin{lemma}
There exists a constant $c > 0$ and $\varepsilon_0 > 0$ (both depending only on $f$) such that for all $\varepsilon \in (0, \varepsilon_0)$,
$$\forall~ |k - \ell| \leq c \cdot \varepsilon^{-2} \quad \qquad I_{k} \cap I_{\ell} = \emptyset.$$
\end{lemma}
\begin{proof} The idea is as follows: if a certain iterate $f^{(k)}(x)$ ends up being close to one of the poles, we would expect $f^{(k+1)}(x)$ be pretty far away from any of the poles. At that point, we can apply Lemma 1 and conclude that for many of the subsequent iterations, that point remains far away from the poles. This can be made precise as follows: for any $\delta>0$ sufficiently small, we see that
\begin{align*}
 f(\beta_j + \delta) &= \beta_j + \delta - \sum_{i=1}^{m} \frac{ \alpha_i}{\beta_j + \delta - \beta_i} \\
 &= - \frac{\alpha_j}{\delta} + \beta_j + \delta -   \sum_{i=1 \atop i \neq j}^{m} \frac{ \alpha_i}{\beta_j + \delta - \beta_i}.
\end{align*}
The first term has a singularity, the remaining terms are all bounded (in a sufficiently small neighborhood of $\beta_j$) because the poles are all distinct. This means that, in a sufficiently small neighborhood $\delta_j$ of $\beta_j$, we have that
$$|f^{(k)}(x)  - \beta_j| \leq \delta_j \implies |f^{(k+1)}(x)| \geq \frac{\alpha_j}{2} \frac{1}{|f^{(k)}(x)  - \beta_j|} \geq \frac{\alpha_j}{2} \frac{1}{\delta_j}.$$
Taking now $\varepsilon_0 = \min \left\{ \delta_1, \dots, \delta_m \right\},$ we deduce that, for any $\varepsilon \in (0, \varepsilon_0)$, we have
$$ \min_{1 \leq i \leq m} | f^{(k)}(x)- \beta_i| \leq \varepsilon  \implies |f^{(k+1)}(x)| \geq \left( \min_{1 \leq i \leq m} \frac{\alpha_i}{2} \right) \frac{1}{\varepsilon}.$$
In other words, $x \in I_k$ implies that $|f^{(k+1)}(x)| \geq c^*/\varepsilon$. At this point, Lemma 1 can be used to conclude that 
$$ \forall~1 \leq t \leq \frac{c}{\varepsilon^2} \qquad |f^{(k+t)}(x)| \geq \frac{c^*}{2 \varepsilon} \gg 1.$$
This then implies that $I_k \cap I_{\ell} = \emptyset$ as long as $|k - \ell| \leq c \varepsilon^{-2}$. 
\end{proof}

\subsection{Glasser's Theorem}
We invoke an amazing result of Glasser \cite{glasser}.
\begin{thm}[Glasser's Master Theorem] Suppose
$$f(x) = x - \sum_{i=1}^{m} \frac{\alpha_i}{x - \beta_i},$$
with $\alpha_i > 0$ and $\beta_i \in \mathbb{R}$. Then, for any absolutely integrable $h:\mathbb{R} \rightarrow \mathbb{R}$,
$$ \emph{p.v.} \int_{\mathbb{R}} h(x) dx =  \emph{p.v.} \int_{\mathbb{R}} h(f(x)) dx.$$
\end{thm}
A particular special case, the identity
$$ \int_{\mathbb{R}} h(x) dx = \int_{\mathbb{R}} h(x - 1/x) dx,$$
was already known to Cauchy \cite{cauchy} in 1823. The same idea also appears in a letter by Schl\"omilch to Liouville \cite{liou} where Schl\"omilch describes that it is also given in his 1848 book \textit{Analytische Studien} as well as in a table of integrals by  Bierens de Haan (\textit{Vous trouverez aussi la formule dans la collection des int\'egrales d\'efinies de M. Bierens de Han} [sic!] \textit{\'a Amsterdam.}). Glasser's theorem is a valuable tool in the evaluation of definite integrals \cite{am}. This magic
identity is also related to another surprising fact, the Boole-Stein-Weiss phenomenon for the Hilbert transform \cite{boole, stein}.
To the best of our knowledge, this is the first time Glasser's magic identity is used in the context of dynamical systems.
\begin{lemma}
There exists $\varepsilon_0 > 0$ (depending only on $f$) such that for all $\varepsilon \in (0, \varepsilon_0)$, 
$$ | I_k| =  \left| \left\{x \in \mathbb{R}: \min_{1 \leq i \leq m} | f^{(k)}(x)- \beta_i| \leq \varepsilon \right\} \right| = m \varepsilon.$$
\end{lemma}
\begin{proof} The argument is clear when $k = 0$ once $\varepsilon$ is sufficiently small (smaller than half the minimal distance between any two distinct poles). We now proceed via induction $k \rightarrow k+1$. We consider the function
$$ h(x) = \begin{cases}
1 \qquad &\mbox{if}~ \min_{1 \leq i \leq m} |x- \beta_i| \leq \varepsilon \\
0 \qquad &\mbox{otherwise.} 
\end{cases}$$
Then $h$ is absolutely integrable (being bounded and compactly supported) and, as we have just seen, for $\varepsilon$ sufficiently small,
$$ \int_{\mathbb{R}} h(x) dx = m \varepsilon.$$
We have
$$ |I_k| = \int_{\mathbb{R}} \chi_{I_k}(x) dx = \int_{\mathbb{R}} h(g^{(k)}(x)) dx.$$
It remains to argue that $h(g^{(k)}(x))$ is absolutely integrable: we note that it is bounded by 1 and, as a consequence of Lemma 1, we have $I_k \subseteq [-c \sqrt{k}, c\sqrt{k}]$ implying that $h \circ g^{(k)}$
 has compact support. Glasser's Theorem applies and 
$$ \int_{\mathbb{R}} h(g^{(k)}(x)) dx = \int_{\mathbb{R}} h(g^{(k)}(g(x))) dx=  \int_{\mathbb{R}} h(g^{(k+1)}(x)) dx = |I_{k+1}|.$$
\end{proof}

\subsection{Conclusion.} With these ingredients in place, we can now finish the argument.
Lemma 2 guarantees that, once $\varepsilon \in (0, \varepsilon_0)$, that
$$ J= \bigcup_{k=1}^{c_2/\varepsilon^{2}} I_k \qquad \mbox{is a union of disjoint sets.}$$
Therefore
$$ |J| = \left| \bigcup_{k=1}^{c_2/\varepsilon^{2}} I_k \right| = \sum_{k=1}^{c_2/ \varepsilon^{2}} |I_k|.$$
At this point, we invoke Lemma 3 and deduce
$$ |J| = \sum_{k=1}^{c_2/ \varepsilon^{2}} |I_k| = \frac{c_2}{\varepsilon^2} m \varepsilon = \frac{c_2 m}{\varepsilon}.$$
The last ingredient consists in applying Lemma 1 once more to deduce that
$$  \bigcup_{k=1}^{c_2/\varepsilon^{2}} I_k \subseteq \left[- \frac{1}{\varepsilon}, \frac{1}\varepsilon\right].$$
This can be seen as follows: Lemma 1 states that if
$$ \min_{1 \leq i \leq m} |x- \beta_i| \geq c_1,$$
then
$$ \min_{1 \leq j \leq c_2 x^2}  \min_{1 \leq i \leq m} |f^{(j)}(x) - \beta_i| \geq \frac{1}{2}  \min_{1 \leq i \leq m} |x- \beta_i|.$$
Setting $c_2 x^2 = c_2/\varepsilon^2$, we see that any initial value $|x| \geq \varepsilon^{-1}$ cannot iterate to a point close to a pole within $c_2/\varepsilon^2$ iteration steps. Setting now $x = 1/\varepsilon$, we see that there are $c_1, c_2, N_0 > 0$ (depending on $f$) so that, for all $N \geq N_0 $
$$ \left| \left\{ x \in [-N, N]:  \min_{1 \leq n \leq c_1 N^2}  \min_{1 \leq i \leq m} |f^{(n)}(x) - \beta_i|  \leq \frac{1}{N} \right\} \right| \geq c_2 N.$$
By not considering the set of points at distance $c_2 N/100 $ from the origin, we deduce
$$ \left| \left\{ \frac{c_2}{100}N \leq |x| \leq N:  \min_{1 \leq n \leq c_1 N^2}  \min_{1 \leq i \leq m} |f^{(n)}(x) - \beta_i|  \leq \frac{1}{N} \right\} \right| \geq \frac{c_2}{2} N.$$
However, for all points in that set, we have that
$$  c_1 N^2 = c_1 \left(\frac{100}{c_2} \frac{c_2}{100}N \right)^2 \leq \frac{10000 c_1^2}{c_2^2} |x|^2$$
which proves the result.


\begin{thebibliography}{10}
\bibitem{am} T. Amdeberhan, M. L. Glasser, M. C. Jones, V., Moll, R. Posey and D. Varela, D. The Cauchy–Schl\"omilch transformation. Integral Transforms and Special Functions, 30 (2019), p. 940-961.

\bibitem{ao} K. Aomoto and P. Forrester, On a Jacobian identity associated with real hyperplane arrangement, Compositio Math. 121 (2000), no. 3, 263--295.

\bibitem{boole} G. Boole, On the comparison of transcendents, with certain applications to the theory of definite integrals, Philos. Trans. R. Soc. Lond. Ser. 147 (1857), p. 745--803.

\bibitem{cauchy} A. L. Cauchy, Sur une formule generale relative a la transformation des integrales simples prises entre les limites $0$ et $\infty$ de la variable, Oeuvres completes, serie 2, Journal de l’Ecole Polytechnique, XIX cahier, tome XIII, 516--519, 1:275–357, 1823

\bibitem{chamberland} M. Chamberland and M. Martelli, Unbounded orbits and binary digits. The Journal of Difference Equations and Applications, 9(7), 687--691.

\bibitem{glasser} M. L. Glasser, A remarkable property of definite integrals, Mathematics of Computation 40 (1983), p. 561--563.

\bibitem{liou}  J. Liouville. Sur l’integrale $\int_0^1 \frac{t^{\mu+1/2} (1-t)^{\mu - 1/2}}{(a + bt  - ct^2)^{\mu+1}}$, 
Extrait d'une lettre de M. O. Schl\"omilch. Extrait d'une lettre de M. A. Cayley. Remarques
de M. Liouville. J. Math. Pures Appl., 2:47--55, 1857.

\bibitem{stein}  E.M. Stein and G. Weiss, An extension of a theorem of Marcinkiewicz and some of its applications, J. Math. Mech. 8 (1959), p. 263--284.

\end{thebibliography}
\end{document}